%% file: ricks-entropy.tex
\documentclass{amsart}

\input{artcfg}

\newcommand{\pifp}{\ensuremath \pi_{\textit{fp}}} %
\newcommand{\piG}{\ensuremath \pi_{\modG{X}}^{X}}
\newcommand{\piGS}{\ensuremath \pi_{\modG{SX}}^{SX}}
\newcommand{\piGSi}{\ensuremath \piGS\phantom{}^{-1}}

\newcommand{\htop}{\ensuremath{h_{\mathrm{top}}}}
\DeclareMathOperator{\Ball}{B}

\DeclareMathOperator{\image}{im}

\newcommand{\eqrefstar}[1]{(\ref*{#1})}

\title[The unique measure of maximal entropy for a CAT(0) space]{The unique measure of maximal entropy for a compact rank one locally CAT(0) space}

\author{Russell Ricks}
\address{Binghamton University, Binghamton, New York, USA}
\email{ricks@math.binghamton.edu}

\thanks{
The author was partially supported by NSF RTG 1045119.}

\begin{document}

\begin{abstract}
Let $X$ be a compact, geodesically complete, locally CAT(0) space such that the universal cover admits a rank one axis.
We prove the Bowen-Margulis measure on the space of geodesics is the unique measure of maximal entropy for the geodesic flow, which has topological entropy equal to the critical exponent of the Poincar\'e series. \end{abstract}

\maketitle

\section{Introduction}

Let $\G$ be a group acting geometrically (that is: properly discontinuously, cocompactly, and isometrically) and freely on a proper, geodesically complete $\CAT(0)$ space $X$ with rank one axis.
The space of unit-speed parametrized geodesics, $SX$, has a natural geodesic flow. The quotient under $\G$ admits a Borel probability measure of full support called the Bowen-Margulis measure \cite{ricks-mixing}.
In this paper we show that the geodesic flow has topological entropy equal to the critical exponent of the Poincar\'e series, and that the Bowen-Margulis measure is (up to rescaling) the unique measure of maximal entropy for the geodesic flow.

\begin{maintheorem}				\label{main manning}
Let $\G$ be a group acting freely geometrically on a proper, geodesically complete $\CAT(0)$ space $X$.
The geodesic flow on $\modG{SX}$ has topological entropy equal to the critical exponent of the Poincar\'e series.
\end{maintheorem}

\begin{maintheorem}				\label{main knieper}
Let $\G$ be a group acting freely geometrically on a proper, geodesically complete $\CAT(0)$ space $X$ with rank one axis.
The Bowen-Margulis measure on $\modG{SX}$ is (up to rescaling) the unique measure of maximal entropy for the geodesic flow on $\modG{SX}$.
\end{maintheorem}

These two theorems generalize well-known results of Manning \cite{manning} and Knieper \cite{knieper98}, respectively, for nonpositively-curved Riemannian manifolds.
The techniques used here are essentially those of the original authors, but the details of the adaptations are finicky enough that we felt they merited being written up properly.
This paper provides those details.

We mention two additional points of interest in the current work.
First, Manning proved an inequality for the topological entropy for all compact Riemannian manifolds; we extend this to all proper, geodesically complete, geodesic metric spaces (\lemref{dG le htop}).
Second, in order to adapt Knieper's proof of the uniqueness of the measure of maximal entropy, we show that $\modG{SX}$ is finite dimensional; this follows from showing that $X$ is finite dimensional.
In particular, every closed ball in a proper, geodesically complete $\CAT(\kappa)$ space ($\kappa \in \R$) is finite dimensional (\lemref{X is fin-dim}).
\footnote{\lemref{X is fin-dim} also follows from [independent work by Lytchak and Nagano](ln19).}

\section{Topological entropy of the geodesic flow}

\subsection{The space of geodesics}

Let $X$ be a proper metric space (i.e.~ all closed metric balls are compact).
A \defn{geodesic} in $X$ is an isometric embedding $v \colon \R \to X$.
A \defn{geodesic segment} is an isometric embedding of a compact interval, and a \defn{geodesic ray} is an isometric embedding of $[0, \infty)$.
We call the space $X$ \defn{geodesic} if for every pair of distinct $x,y \in X$ there is a geodesic segment in $X$ with endpoints $x$ and $y$.
We call $X$ \defn{geodesically complete} if every geodesic segment in $X$ can be extended to a geodesic $v \colon \R \to X$.

Denote by $SX$ the space of all geodesics $\R \to X$, where $SX$ is endowed with the compact-open topology.
This space is metrizable with metric
\begin{equation*}
d_{SX}(v,w) = \sup_{t \in \R} e^{-\abs{t}} d_X(v(t), w(t)).
\end{equation*}
Under this metric, $SX$ is a proper metric space; moreover, the \defn{footpoint projection} $\pifp \colon SX \to X$ given by $\pifp(v) = v(0)$ is $1$-Lipschitz.
Note that $\pifp$ is a proper map by the Arzel\`a-Ascoli Theorem since $X$ is proper.
The \defn{geodesic flow} $g^t$ on $SX$, defined by $(g^t v)(r) = v(t + r)$, is $e^{\abs{t}}$-Lipschitz for all $t \in \R$.

The following lemma gives uniform control over the compact-open topology.

\begin{lemma}					\label{pre-5.2}
For every $\delta > 0$ there exists $r \ge 0$ such that for every $L$-Lipschitz map $f \colon \R \to \R$,
if $e^{-\abs{t}} f(t) \le \delta$ for all $t \in [-r,r]$,
then $e^{-\abs{t}} f(t) \le \delta$ for all $t \in \R$.
\end{lemma}

\begin{proof}
Let $r = \max \set{0, \log(\frac{L}{\delta})}$.
Suppose that $t_0 > r$ satisfies $f(t_0) > \delta e^{\abs{t_0}}$ but $f(t) \le \delta e^{\abs{t}}$ for all $t \in [-r,r]$.
Then
\[f(t_0) - f(r)
> \delta (e^{t_0} - e^{t})
= \delta \int_{r}^{t_0} e^t \, dt
> \delta e^{r} (t_0 - r)
\ge L (t_0 - r),\]
contradicting our hypothesis that $f$ is $L$-Lipschitz.
Similarly for $t_0 < -r$.
\end{proof}

\begin{corollary}				\label{uniform compact-open constants}
For every $\delta > 0$ there is some $r > 0$ such that for all $v,w \in SX$, if $d_{X}(v(t),w(t)) < \delta$ for all $t \in [-r,r]$, then $d_{SX}(v,w) < \delta$.
\end{corollary}

Now let $\G$ be a group acting properly discontinuously, by isometries on $X$.
The $\G$-action on $X$ naturally induces a properly discontinuous, isometric action on $SX$, which is cocompact if and only if the action on $X$ is.
Since the $\G$-action on $SX$ commutes with the geodesic flow $g^t$ on $SX$, $g^t$ descends to a flow $g^t_\G$ on the quotient $\modG{SX}$.
We will write $\piGS \colon SX \to \modG{SX}$ for the canonical projection map.

\subsection{Topological entropy}

We now define topological entropy and mention some basic properties.
A good reference is \cite{walters}. 
Let $X$ be a metric space and $f \colon X \to X$ be continuous.

For $Y \subseteq X$, $n \in \N$, and $\epsilon > 0$, we say a subset $A \subseteq Y$ \defn{$(n,\epsilon)$-spans} $Y$ (with respect to $f$) if for all $y \in Y$ there exists $a \in A$ such that for all $i$ with $0 \le i < n$, we have $d(f^i(y),f^i(a)) \le \epsilon$.
Let $r_n(f,Y,\epsilon)$ be the minimum number of elements in an $(n,\epsilon)$-spanning set of $Y$ under $f$, and define the \defn{topological entropy}
\begin{gather*}
\htop(f,Y) = \lim_{\epsilon \to 0^+}
\limsup_{n \to \infty} \frac{1}{n} \log r_n(f,Y,\epsilon) \\
\text{and} \quad
\htop(f) = \sup \setp{\htop(f,K)}{K \subseteq Z \text{ compact}}.
\end{gather*}

It is a standard fact (see e.g.~ \cite[Corollary 7.5.1]{walters}) that when $f$ is uniformly continuous, we may use arbitrarily small compact sets $K$ when calculating topological entropy; that is, for all $\delta > 0$ we have
\[\htop(f) = \sup \setp{\htop(f,K)}{K \subseteq Z \text{ compact and } \diam K < \delta}.\]

It is often useful to use the notion of a separated set, which is in some sense dual to the notion of a spanning set.
For $n \in \N$ and $\epsilon > 0$, we say a subset $A \subseteq X$ \defn{$(n,\epsilon)$-separated} (with respect to $f$) if for all distinct $x,y \in A$, some $i$ with $0 \le i < n$ satisfies $d(f^i(x),f^i(y)) > \epsilon$.
We call a set $A \subseteq X$ \defn{$\epsilon$-separated} if it is $(1,\epsilon)$-separated, i.e.~ $d(x,y) > \epsilon$ for all distinct $x,y \in A$; likewise a set $A \subseteq Y$ \defn{$\epsilon$-spans} $Y$ if $d(y,A) \le \epsilon$ for all $y \in Y$.
Write $s_n(f,Y,\epsilon)$ for the maximum number of elements in an $(n,\epsilon)$-separated subset of $Y$ under $f$.
Then
\begin{equation*}
r_n(f,\epsilon,Y)
\le s_n(f,\epsilon,Y)
\le r_n(f,\frac{\epsilon}{2},Y)
\end{equation*}
and therefore
\begin{equation*}
\htop(f,Y)
:= \lim_{\epsilon \to 0^+}
\limsup_{n \to \infty} \frac{1}{n} \log r_n(f,Y,\epsilon)
= \lim_{\epsilon \to 0^+}
\limsup_{n \to \infty} \frac{1}{n} \log s_n(f,Y,\epsilon).
\end{equation*}
Thus we could use either spanning sets or separated sets in the definition of $\htop(f)$.

It is well known (see e.g.~ \cite[Theorem 7.10]{walters})
that when $f \colon X \to X$ is uniformly continuous, then $\htop(f^k) = k \htop(f^k)$ for all $k > 0$.
Thus, for a continuous flow $\phi = (\phi^t)_{t \in \R} \colon X \to X$, we define the \defn{topological entropy} of $\phi$ to be the topological entropy of the time-one map $\phi^1$, i.e.~ $\htop(\phi) := \htop(\phi^1)$.
Then $\htop(\phi^r) = r \htop(\phi)$ for all $r > 0$.
We remark that although in general $\htop(\phi^{-1}) \neq \htop(\phi)$,
when $\phi$ is the geodesic flow on some $SX$ then $\htop(\phi^r) = \abs{r} \htop(\phi^r)$ for all $r \in \R$.

\begin{definition}
Let $\G$ be a group acting properly discontinuously, isometrically on a proper metric space $X$.
Let $\piG \colon X \to \modG{X}$ be the canonical projection.
Define \begin{equation*}
\injrad(X,\G) = \inf_{x \in X} \sup \smallsetp{r \ge 0}{\piG \res{\Ball(x,r)} \text{ is injective}}.
\end{equation*}
Note if $x,y \in X$ have $d(x,y) < \injrad(X,\G)$, then $d(\piGS(x),\piGS(y)) = d(x,y)$.
\end{definition}

The next lemma follows from \cite[Theorem 8.12]{walters}.

\begin{lemma}					\label{entropy of covers}
Let $\G$ act isometrically on a proper metric space $X$, and let $f \colon X \to X$ be a $\G$-equivariant and uniformly continuous map.
If $\injrad(X,\G) > 0$, then the factor map $f_\G \colon \modG{X} \to \modG{X}$ satisfies $\htop(f_\G) = \htop(f)$.
\end{lemma}

Thus the geodesic flow $g^t$ on $SX$ for a proper, geodesically complete, geodesic metric space $X$, under an isometric action by $\G$ with $\injrad(X,\G) > 0$, must satisfy
\[\htop(g^t) = \htop(g_\G^t).\]

\subsection{A general inequality}

Let $\G$ be a group acting properly discontinuously and isometrically on a proper metric space $X$.
The critical exponent
\[\delta_\G = \inf \smallsetp{s \ge 0}{\sum_{\g \in \G} e^{-s d(p, \g q)} < \infty}\]
of the Poincar\'e series for $\G$ does not depend on choice of $p$ or $q$.
We remark that $\delta_\G < \infty$ whenever $\G$ is finitely generated.

\begin{lemma}	[cf.~ Theorem 1 of \cite{manning}]	\label{dG le htop}
Let $\G$ be a group acting properly discontinuously and isometrically on a proper, geodesically complete, geodesic metric space $X$.
Then the geodesic flow $\phi = (\phi^t)_{t \in \R} = (g^t)_{t \in \R}$ on $SX$ satisfies $\htop(\phi) \ge \delta_\G$.
\end{lemma}

In particular, if $\injrad(X,\G) > 0$ then the factor flow $\phi_\G$ on $\modG{SX}$ satisfies
\[\htop(\phi_\G) = \htop(\phi) \ge \delta_\G.\]

\begin{proof}
The case $\delta_\G = 0$ is trivial, so assume $\delta_\G > 0$.
Fix $x \in X$, and find $\epsilon > 0$ small enough that for all $\g \in \G$, either $\g x = x$ or $d(x, \g x) > 3 \epsilon$.
Let $\alpha > 0$ be arbitrary.
For each $r > 0$, let $S(r) = \Ball(x, r + \epsilon) \setminus \Ball(x, r)$ and $V_S(r) = \card{\setp{\g \in \G}{\g x \in S(r)}}$.
It follows from the definition of $\delta_\G$ that there is a sequence $r_k \to \infty$ such that $V_S(r_k) \ge e^{(\delta_\G - \alpha)r_k}$ for all $k$; we may assume each $r_k = \epsilon n_k$ for some $n_k \in \N$.
Now for each point $y \in S(r_k) \cap \G x$ choose $v_{x,y} \in SX$ such that $v_{x,y}(0) = x$ and $v_{x,y}(d(x,y)) = y$.
Then the set $A_{r_k} := \setp{v_{x,y}}{y \in S(r_k) \cap \G x}$ is $(r_k,\epsilon)$-separated because each pair of distinct $v_{x,y}, v_{x,y'} \in A_{r_k}$ has $d(g^{r_k} v_{x,y}, g^{r_k} v_{x,y'}) \ge d(v_{x,y}(r_k), v_{x,y'}(r_k)) \ge d(y, y') - d(y, v_{x,y}(r_k)) - d(y', v_{x,y'}(r_k)) > 3\epsilon - \epsilon - \epsilon = \epsilon$.
Also, $A_{r_k} \subseteq \pifp^{-1} \set{x}$, which is compact because $X$ is proper.
Thus $\htop(\phi) \ge \limsup_{r_k \to \infty} \frac{1}{r_k} \log \card{A_{r_k}} = \limsup_{k \to \infty} r_k^{-1} \log V_S(r_k) \ge \delta_\G - \alpha$.
Since $\alpha > 0$ was arbitrary, $\htop(\phi) \ge \delta_\G$.
\end{proof}

\subsection{Nonpositive curvature}

A metric space $X$ is called \defn{CAT(0)} if $X$ is geodesic and satisfies the \defn{CAT(0) inequality}: for every triple of distinct points $x,y,z \in X$, points on the geodesic triangle $\triangle(x,y,z)$ are no further apart than the corresponding points on a triangle in Euclidean $\R^2$ with the same edge lengths.
This generalizes nonpositive curvature from Riemannian manifolds.
For more on $\CAT(0)$ spaces, see \cite{ballmann} or \cite{bridson}.
An important consequence of the $\CAT(0)$ inequality is that every $\CAT(0)$ metric is \defn{convex}---that is, for every pair of geodesics (or even geodesic segments) $v,w$ in $X$, the function $t \mapsto d_X (v(t),w(t))$ is convex.

\begin{lemma}	[cf.~ Theorem 2 of \cite{manning}]	\label{htop le dG}
Let $\G$ be a group acting geometrically on a proper, geodesically complete,
convex (e.g.~ $\CAT(0)$) geodesic metric space $X$.
Then the geodesic flow $\phi = (\phi^t)_{t \in \R} = (g^t)_{t \in \R}$ on $SX$ satisfies $\htop(\phi) \le \delta_\G$.
\end{lemma}

\begin{proof}
Choose a basepoint $p \in X$, and find $R > 0$ such that the compact set $K_1 := \cl{\Ball}(p,R)$ in $X$ satisfies $\G K_1 := \bigcup_{\g \in \G} \g K_1 = X$.
Let $L \subset SX$ be an arbitrary compact set.
Now $\pifp(L)$ is compact in $X$, so there is some $c \ge 1$ such that $\pifp(L)$ is a subset of the compact set $K := \cl{\Ball}(p,cR)$.
Hence $L$ is a subset of the compact set $K' := \pifp^{-1} K$.
Since every $(n,\epsilon)$-separated subset of $L$ is an $(n,\epsilon)$-separated subset of $K'$, we see that $\htop(\phi, L) \le \htop(\phi, K')$.
We will prove $\htop(\phi, K') \le \delta_\G$.

Let $\alpha, \epsilon, n > 0$ be arbitrary.
It follows from the definition of $\delta_\G$ that there is some $r_0 > 0$ such that for all $r \ge r_0$,
we have $\card{\setp{\g \in \G}{\g p \in \Ball(p,r)}}
\le e^{(\delta_\G + \alpha)r}$.
Choose $a_{\epsilon} \ge r_0$ such that for all $v,w \in SX$, if $d(v(t),w(t)) \le \epsilon$ for all $t \in [-a_{\epsilon}, a_{\epsilon}]$, then $d(v,w) \le \epsilon$.

Choose a minimal $\frac{\epsilon}{3}$-spanning set $Q \subset K$ for $K$.
Let $B_1 = \Ball(p, a_{\epsilon} + \diam K)$.
By assumption on $K$, the set $B_1$ is covered by $\G$-translates of $K$.
Every such $\G$-translate of $K$ is contained in $D_1 = \Ball(p, a_{\epsilon} + 2\diam K)$.
Since $a_{\epsilon} \ge r_0$, we find
\[\card{\setp{\g \in \G}{\g p \in D_1}}
\le e^{(\delta_\G + \alpha)(a_{\epsilon} + 2\diam K)}\]
and therefore $B_1$ is covered by at most $e^{(\delta_\G + \alpha)(a_{\epsilon} + 2\diam K)}$ $\G$-translates of $K$.
Thus $B_1$ contains an $\frac{\epsilon}{3}$-spanning set $Q_1$ (with points in $\G Q$) of cardinality at most $e^{(\delta_\G + \alpha)(a_{\epsilon} + 2\diam K)} \card{Q}$.
Similarly, let $B_2 = \Ball(p, a_{\epsilon} + n + \diam K)$; the above argument shows $B_2$ contains an $\frac{\epsilon}{3}$-spanning set $Q_2$ of cardinality at most
$e^{(\delta_\G + \alpha)(a_{\epsilon} + n + 2\diam K)} \card{Q}$.

Note that $K' := \pifp^{-1} (K) \subset SX$ is compact by properness of $\pifp$.
For each $x \in Q_1$ and $y \in Q_2$, choose some $v_{x,y} \in SX$ such that $v_{x,y}(-a_{\epsilon}) = x$ and $v_{x,y}(d(x,y) - a_{\epsilon}) = y$.
We claim the set $A_{n} = \setp{v_{x,y}}{x \in Q_1 \text{ and } y \in Q_2}$ is an $(n,\epsilon)$-spanning set for $K'$.
Let $w \in K'$ be arbitrary.
Since $w(0) \in K$, we know $w(-a_{\epsilon}) \in B_1$ and $w(n + a_{\epsilon}) \in B_2$.
Thus there exist $x \in Q_1$ and $y \in Q_2$ such that $d(x, w(-a_{\epsilon})) \le \frac{\epsilon}{3}$ and $d(y, w(n + a_{\epsilon})) \le \frac{\epsilon}{3}$.
Then
$d(y, v_{x,y}(n + a_{\epsilon}))
= \abs{d(x,y) - (n + 2 a_{\epsilon})}
= \abs{d(x,y) - d(w(-a_{\epsilon}), w(n + a_{\epsilon}))} \le 
\frac{2\epsilon}{3}$
by the triangle inequality,
and we conclude $d(v_{x,y}(n + a_{\epsilon}), w(n + a_{\epsilon})) \le \epsilon$.
By convexity of the metric on $X$, $d(v_{x,y}(t), w(t)) \le \epsilon$ for all $t \in [-a_{\epsilon}, n + a_{\epsilon}]$.
Therefore, $d(g^t v_{x,y}, g^t v_{x',y'}) \le \epsilon$ for all $t \in [0, n]$ by choice of $a_{\epsilon}$.
This proves $A_{n}$ is an $(n,\epsilon)$-spanning set for $K'$, as claimed.

Thus we see that, since $\card{Q}$ and $\card{Q_1}$ do not depend on $n$,
\begin{equation*}
\limsup_{n \to \infty} \frac{1}{n} \log r_n(\phi,K',\epsilon)
\le \limsup_{n \to \infty} \frac{1}{n} \log \card{A_{n}}
= \limsup_{n \to \infty} \frac{1}{n} \log \card{Q_2}
= \delta_\G + \alpha.
\end{equation*}
Since $\alpha, \epsilon > 0$ were arbitrary, $\htop(\phi, K') \le \delta_\G$.
Thus $\htop(\phi, L) \le \htop(\phi, K') \le \delta_\G$.
Since $L$ was arbitrary, $\htop(\phi) \le \delta_\G$.
\end{proof}

Combining Lemmas \ref{entropy of covers}, \ref{dG le htop}, and \ref{htop le dG} gives us \thmref{main manning}.

\section{Finite Dimension}

We now show $\modG{SX}$ is finite dimensional;
we will use this to prove \corref{knieper98 3.4}.

Every complete $\CAT(0)$ space $X$ has an \defn{ideal boundary}, written $\bd X$, obtained by taking equivalence classes of asymptotic geodesic rays.
The compact-open topology on the set of rays induces a topology on $\bd X$, called the \defn{cone} or \defn{visual} topology.
When $X$ is proper, both $\bd X$ and $\cl{X} = X \cup \bd X$ are compact metrizable spaces.

By classical theory (see, e.g.~ \cite[Theorem 1.7.7]{engelking}) for a separable metric space, the Lebesgue covering dimension, small inductive dimension, and large induction dimension are all equal.
Moreover, Kleiner remarks \cite[p. 412]{kleiner} that the geometric dimension of a separable $\CAT(0)$ space $X$ equals the Lebesgue covering dimension.

Throughout this section $X$ will be a proper $\CAT(0)$ space, and thus all the spaces $X$, $SX$, $\modG{X}$, $\modG{SX}$, $\bd X$, and $\cl{X}$ are separable metric spaces.
Therefore, we will simply write \defn{finite dimensional} to mean any of the equivalent definitions.

\begin{lemma}					\label{X is fin-dim}
Let $X$ be a proper, geodesically complete $\CAT(0)$ space which admits a cocompact action by isometries.
Then $X$ is finite dimensional.
\end{lemma}

\begin{proof}
Suppose, by way of contradiction, that $X$ has infinite geometric dimension.
By Kleiner \cite[Theorem A]{kleiner}, this means for every $k \in \N$ there exist $p_k \in X$ and sequences $R_j \to 0, S_j \subset X$ such that $d(S_j, p_k) \to 0$, and $\frac{1}{R_j} S_j$ converges to the unit ball $\Ball(1) \subset \R^k$ in the Gromov--Hausdorff topology.
By cocompactness we find some $p \in X$ such that for every $k \in \N$ there exist sequences $R_j \to 0, S_j \subset X$ such that $d(S_j, p) \to 0$, and $\frac{1}{R_j} S_j$ converges to the unit ball $\Ball(1) \subset \R^k$ in the Gromov--Hausdorff topology.
Thus by \cite[Theorem 9.4]{ricks-mixing} we find round spheres of dimension $k$ in the link of $X$ at $p$, for all $k$.
Each such $k$-dimensional sphere contains $2k$ points that are $\frac{\pi}{2}$-separated, and by geodesic completeness we can extend these directions to obtain a $\frac{\pi}{2}$-separated subset of $2k$ points in the metric sphere of radius $1$ about $p$ in $X$.
This holds for all $k$, violating properness of the metric on $X$.
Therefore, $X$ must have finite geometric dimension.
\end{proof}

\begin{remark}
The above proof works, with minor modification, when $X$ is assumed to be a proper, geodesically complete $\CAT(\kappa)$ space which admits a cocompact action by isometries, for any $\kappa \in \R$.
Additional minor modifications show $X$, even when not cocompact, is still locally finite dimensional.
\end{remark}

The following is due to Eric Swenson \cite{swenson99}.

\begin{theorem} [Theorem 12 of \cite{swenson99}]
If $X$ is a proper $\CAT(0)$ space which admits a cocompact action by isometries, then $\bd X$ is finite dimensional.
\end{theorem}

\begin{lemma}
Let $X$ be a proper $\CAT(0)$ space.
The map $SX \to X \times \dbX$ given by $v \mapsto (\pi(v), \emap(v)) = (v(0), v^-, v^+)$ is a topological embedding.
\end{lemma}

\begin{lemma}
Let $X$ be a proper, geodesically complete $\CAT(0)$ space which admits a cocompact action by isometries.
Then $SX$ is finite dimensional.
\end{lemma}

\begin{proof}
The product $X \times \dbX$ of finite-dimensional spaces is finite dimensional \cite[Theorem 1.5.16]{engelking}.
Since $SX$ embeds into $X \times \dbX$, it is also finite dimensional \cite[Theorem 1.1.1]{engelking}.
\end{proof}

\begin{corollary}				\label{SXG is fin-dim}
Let $\G$ be a group acting freely geometrically on a proper, geodesically complete $\CAT(0)$ space $X$.
Then $\modG{X}$ and $\modG{SX}$ are finite dimensional.
\end{corollary}

\begin{proof}
The projections $X \to \modG{X}$ and $SX \to \modG{SX}$ are open with discrete fibers, so $\dim(X) = \dim(\modG{X})$ and $\dim(SX) = \dim(\modG{SX})$ by \cite[Theorem 1.12.7]{engelking}.
\end{proof}

\section{Review of Bowen-Margulis measures}

STANDING HYPOTHESIS:
For the remainder of the paper, let $\G$ be a group acting freely geometrically on a proper, geodesically complete, $\CAT(0)$ space $X$ with a rank one axis.
To simplify the exposition, we will exclude the trivial case $X = \R$ by always assuming $\bd X$ has at least three points.

\subsection{More on $\CAT(0)$ spaces}

We have a natural endpoint projection $\emap \colon SX \to \dbX$ defined by $\emap(v) = (v^-, v^+) := (\lim_{t \to -\infty} v(t), \lim_{t \to +\infty} v(t))$.
And in fact $v \in SX$ is parallel to $w \in SX$ if and only if $\emap(v) = \emap(w)$.
We will also use the map $\pi_p \colon SX \to \dbX \times \R$ given by
$\pi_p(v) = (v^-, v^+, b_{v^-} (v(0), p))$.
Define the \defn{cross section} of $v \in SX$ to be $CS(v) = \pi_p^{-1} \set{\pi_p(v)}$, and the \defn{width} of a geodesic $v \in SX$ to be $\width(v) = \diam CS(v)$.
The width of $v$ is in fact the maximum width of a flat strip $\R \times [0, R]$ in $X$ parallel to $v$.

We call a geodesic in a $\CAT(0)$ space \defn{rank one} if $\width(v) < \infty$, and \defn{zero width} if $\width(v) = 0$.
Write $\Reg \subseteq SX$ for the set of rank one geodesics, and $\Zerowidth \subseteq \Reg$ for the set of zero-width geodesics.
The following lemma describes an important aspect of the geometry of rank one geodesics in a $\CAT(0)$ space.

\begin{lemma}[Lemma III.3.1 in \cite{ballmann}]\label{Ballmann's lemma}
Let $v \in SX$ have $\width(v) < R$ for some $R > 0$.
There are neighborhoods $U$ and $V$ in $\cl{X}$ of $v^-$ and $v^+$ such that for any $\xi \in U$ and $\eta \in V$, there is a geodesic $w$ joining $\xi$ to $\eta$.
For any such $w$, we have $d(v(0), \image(w)) < R$; in particular, 
we may assume $\width(w) < R$ for all such $w$.
\end{lemma}

The \defn{Tits metric} $\dT$ on $\bd X$ induces a topology that is finer (usually strictly finer) than the visual topology.
The Tits metric is complete $\CAT(1)$, and measures the asymptotic angle between geodesic rays in $X$.
In fact, $(\xi,\eta) \in \emap(\Reg)$ if and only if $\dT(\xi,\eta) > \pi$.

A geodesic $v \in SX$ is \defn{axis} of an isometry $\g \in \Isom X$ if $\g$ translates along $v$, i.e., $\g v = g^t v$ for some $t > 0$.
We note the $\G$-action on $X$ also naturally extends to an action by homeomorphisms on $\cl{X}$ (and therefore on $\bd X$).
It is well-known that if $X$ has a rank one axis, then the $\G$-action on $\bd X$ is minimal.

\subsection{Measures}
\label{subsec: measures}

We recall the measures constructed in \cite{ricks-mixing}.

For $\xi \in \bd X$ and $p,q \in X$, let $b_{\xi} (p, q)$ be the Busemann cocycle
\[b_{\xi} (p, q) = \lim_{t \to \infty} \left[ d([q,\xi)(t), p) - t \right].\]
These functions are $1$-Lipschitz in both variables and satisfy the cocycle property $b_{\xi} (x, y) + b_{\xi} (y, z) = b_{\xi} (x, z)$.  Furthermore, $b_{\gamma \xi} (\gamma x, \gamma y) = b_{\xi} (x, y)$ for all $\gamma \in \Isom X$.

Since $\G$ acts geometrically on $X$, $\G$ is finitely generated and therefore $\delta_\G < \infty$.
Thus Patterson's construction yields a conformal density $\family{\mu_p}_{p \in X}$ of dimension $\delta_\G$ on $\bd X$, called the \defn{Patterson-Sullivan} measure.

\begin{definition}				\label{conformal density}
A \defn{conformal density of dimension $\delta$} is a family $\family{\mu_p}_{p \in X}$ of equivalent finite Borel measures on $\bd X$, such that for all $p,q \in X$ and $\g \in \G$:
\begin{enumerate}
\item \label{equivariance}
the pushforward $\gamma_* \mu_p = \mu_{\gamma p}$ and
\item \label{Radon-Nikodym}
the Radon-Nikodym derivative $\frac{d\mu_q}{d\mu_p}(\xi) = e^{-\delta b_\xi (q, p)}$.
\end{enumerate}
\end{definition}

Now fix $p \in X$.
For $(v^-,v^+) \in \emap(SX)$, define $\beta_p \colon \emap(SX) \to \R$ by $\beta_p (v^-, v^+) = (b_{\xi} + b_{\eta}) (v(0), p)$; this does not depend on choice of $v \in \emap^{-1}(v^-,v^+)$.
The measure $\mu$ on $\dbX$ defined by
\[d\mu (\xi, \eta)
= e^{-\delta_\G \beta_p (\xi, \eta)} d\mu_p (\xi) d\mu_p (\eta)\]
is $\G$-invariant and does not depend on choice of $p \in X$; it is called a \emph{geodesic current}.

The \defn{Bowen-Margulis} measure $m$ is a Radon measure on $SX$ that is invariant under both $g^t$ and $\G$, constructed as follows:
The measure $\mu \times \lambda$ on $\dbX \times \R$, where $\lambda$ is Lebesgue measure, is supported on $\emap(\Zerowidth) \times \R$.
Then $\pi_p \colon SX \to \dbX \times \R$ is seen to restrict a homeomorphism from $\Zerowidth$ to $\emap(\Zerowidth) \times \R$, hence $m = \mu \times \lambda$ may be viewed as a Borel measure on $SX$.

The Bowen-Margulis measure $m$ has a quotient measure $m_\G$ on $\modG{SX}$.
Since we assume $\G$ acts freely on $X$ (and therefore on $SX$), $m_\G$ can be described as follows:  Whenever $A \subset SX$ is a Borel set on which $\piGS$ is injective, $m_\G(\piGS A) = m(A)$.
By proper normalization of $\mu_p$, we may (and will) assume $\norm{m_\Gamma} := m_\G(\modG{SX}) = 1$.

\section{Shadows}

Define \defn{shadows} in $\cl{X}$ as follows:
Let $p \in X$.
For $x \in X$ and $\xi \in \bd X$, write
\begin{gather*}
\pr_x(p) = \setp{v^+ \in \bd X}{v \in \pifp^{-1}(p) \text{ satisfies } v(-d(p,x)) = x} \\
\text{and} \quad
\pr_{\xi}(p) = \setp{v^+ \in \bd X}{v \in \pifp^{-1}(p) \text{ satisfies } v^- = \xi}.
\end{gather*}

We will need some estimates on the measure of these shadows.
We present here an adapted argument from \cite[Proposition 2.3]{knieper98}.

\begin{lemma}					\label{knieper98 2.3 2}
For each $\xi \in \cl{X}$ and $p \in X$ there exist $R > 0$, $\eta \in \bd X$, and an open neighborhood $U \times V$ of $(\xi,\eta)$ in $\cl{X} \times \bd X$ such that $V \subseteq \pr_{\xi'}(\Ball(p,R))$ for all $\xi' \in U$.
\end{lemma}

\begin{proof}
Let $p \in X$ and $\xi \in \cl{X}$.
If $\xi \in X$ then $\pr_{\xi}(\Ball(p,R))$ is open (and nonempty) for all $R > 0$ by definition of the cone topology.
And for all $r > 0$, the open set $U := \Ball(\xi,r) \subset X$ satisfies
$\bigcap_{y \in U} \pr_{y}(\Ball(p,R+r)) \supset \pr_{\xi}(\Ball(p,R))$
by convexity.
So let $\xi \in \bd X$.
Choose a Tits-isolated point $\eta \neq \xi$ in $\bd X$.
Then there exists $v \in \Reg$ with $(v^-,v^+) = (\xi,\eta)$.
By \lemref{Ballmann's lemma} there exists an open neighborhood $U \times V$ of $(v^-,v^+)$ in $\dbX$ such that for all $(\xi',\eta') \in U \times V$, the set $\Reg \cap \emap^{-1}(\xi',\eta')$ is not empty, and there is some $r > 0$ (not depending on $(\xi',\eta')$) such that every $w \in \emap^{-1}(\xi',\eta')$ passes through $\Ball(v(0),r)$.
Putting $R = r + d(p, v(0))$, we see that $V \subseteq \pr_{\xi'}(\Ball(p,R))$ by the triangle inequality for all $\xi' \in U$.
\end{proof}

\begin{lemma}					\label{knieper98 2.3 5}
There are constants $R_0, \ell > 0$ such that
$\mu_p(\pr_x(\Ball(p,R_0))) \ge \ell$
for all $x \in \cl{X}$ and $p \in X$.
\end{lemma}

\begin{proof}
For each $p \in X$, by compactness of $\cl{X}$ we obtain by \lemref{knieper98 2.3 2} a finite collection of nonempty open sets $V_i \subset \bd X$ and constants $R_i > 0$ such that for every $x \in \cl{X}$, some $V_i \subseteq \pr_{x}(\Ball(p,R_i))$.
Now fix a compact set $K \subset X$ such that $X = \G K := \bigcup_{\g \in \G} \g K$, and fix $p_0 \in K$.
Set $r_0 = \max R_i$, $R_0 = \diam K + r_0$, and $\ell = e^{-\delta_\G \diam K} \min \mu_{p_0}(V_i)$.
Since the $\G$-action on $\bd X$ is minimal, $\supp \mu_{p_0} = \bd X$.
Thus $\ell > 0$.
We now have
$\pr_{x}(\Ball(p,R_0)) \supseteq \pr_{x}(\Ball(p_0,r_0))$
for all $x \in \cl{X}$, and thus
\[\mu_{p}(\pr_{x}(\Ball(p,R_0)))
\ge \mu_{p}(V_i)
\ge \ell.\]
Now let $x \in \cl{X}$ and $p \in X$ be arbitrary.
Find $\g \in \G$ such that $\g p \in K$.
Then
\[\mu_p(\pr_{x} \Ball(p, R_0))
= \mu_{\g p}(\g \pr_{x} \Ball(p, R_0))
= \mu_{\g p}(\pr_{\g x} \Ball(\g p, R_0))
\ge \ell. \qedhere\]
\end{proof}

\begin{lemma}					\label{knieper98 2.3 6}
There is a constant $b > 1$ such that for all $p,x \in X$ and $\xi \in \pr_x(p)$,
\begin{equation*}
\frac{1}{b} e^{-\delta_\G d(p,x)} \le \mu_p(\pr_\xi(\Ball(x,R_0))) \le b e^{-\delta_\G d(p,x)},
\end{equation*}
where $R_0 > 0$ is the constant from \lemref{knieper98 2.3 5}.
\end{lemma}

\begin{proof}
Let $p,x \in X$ and $\xi \in \pr_x(p)$.
Then (reversing the usual geodesic) there is some $v \in SX$ such that $v(0) = p$, $v(r) = x$, and $v^- = \xi$, where $r = d(p,x)$.

Let $\eta \in \pr_{\xi}(\Ball(x,R_0))$, and find $w \in SX$ such that $(w^-,w^+) = (\xi,\eta)$ and $w(r) \in \Ball(x,R_0)$.
Write $y = w(r)$ and $q = w(0)$.
By convexity of $t \mapsto d(v(t),w(t))$, we find $d(p,q) \le d(x,y) < R_0$.
By $1$-Lipschitzness of $b_{\eta}$ in each variable, we see from
$b_{\eta}(q,y) = r$ that $b_{\eta}(p,x) \in (r - 2R_0, r + 2R_0)$
for all $\eta \in \pr_{\xi}(\Ball(x,R_0))$.

Hence
\[\mu_p(\pr_{\xi}(\Ball(x,R_0))) = \int_{\pr_{\xi}(\Ball(x,R_0))} e^{-\delta_\G b_{\eta}(p,x)} \, d\mu_x(\eta)\]
lies between $e^{-2\delta_\G R_0} e^{-\delta_\G r} \mu_x(\pr_{\xi}(\Ball(x,R_0)))$ and $e^{2\delta_\G R_0} e^{-\delta_\G r} \mu_x(\pr_{\xi}(\Ball(x,R_0)))$.
This gives us the bounds
\begin{align*}
\ell e^{-2\delta_\G R_0}
\le e^{-2\delta_\G R_0} \mu_x(\pr_{\xi}(\Ball(x,R_0)))
&\le e^{\delta_\G r} \mu_p(\pr_{\xi}(\Ball(x,R_0))) \\
&\le e^{2\delta_\G R_0} \mu_x(\pr_{\xi}(\Ball(x,R_0)))
\le e^{2\delta_\G R_0} \norm{\mu_x},
\end{align*}
and therefore
\begin{equation*}
\ell e^{-2\delta_\G R_0} e^{-\delta_\G d(p,x)}
\le \mu_p(\pr_{\xi}(\Ball(x,R_0)))
\le \norm{\mu_x} e^{2\delta_\G R_0} e^{-\delta_\G d(p,x)}.
\end{equation*}
Notice $\norm{\mu_x} \le e^{2\delta_\G \diam(\modG{X})} \norm{\mu_q}$ for all $q \in X$ by \defref{conformal density}.
\end{proof}

\section{Maximal entropy}

Let $\nu$ be a probability measure on a space $Z$.
The \defn{entropy} of a measurable partition $\mathcal{A} = \set{A_1, \dotsc, A_m}$ of $Z$ is
\[H_{\nu}(\mathcal{A})
= \sum_{i=1}^{m} -\nu(A_i) \log \nu(A_i).\]
Let $T \colon Z \to Z$ be a measure-preserving transformation.
For the partitions
\[\mathcal{A}_{T}^{(n)}
:= \setp{A_{j_1} \cap T^{-1} A_{j_2} \cap \dotsb \cap T^{-(n-1)} A_{j_{n-1}}}{1 \le j_1, j_2, \dotsc, j_{n-1} \le m},\]
$n \mapsto \frac{1}{n} H_{\nu}(\mathcal{A}_T^{(n)})$ is a subadditive function.
Hence $\frac{1}{n} H_{\nu}(\mathcal{A}_T^{(n)})$ decreases to a limit
\[h_{\nu}(T,\mathcal{A})
= \lim_{n \to \infty} \frac{1}{n} H_{\nu}(\mathcal{A}_T^{(n)})
= \inf_{n \in \N} \frac{1}{n} H_{\nu}(\mathcal{A}_T^{(n)}),\]
called the \defn{entropy of $T$ with respect to $\mathcal{A}$}.
The \defn{measure-theoretic entropy} of $T$ is
\[h_{\nu}(T) = \sup_{\mathcal{A}} h_{\nu}(T,\mathcal{A}).\]

As with topological entropy, $h_{\nu}(T^k) = k h_{\nu}(T^k)$ for all $k > 0$, and in fact if $T$ is invertible then $h_{\nu}(T^k) = \abs{k} h_{\nu}(T^k)$ for all $k \in \Z$.
Thus, for a measure-preserving flow $\phi = (\phi^t)_{t \in \R} \colon Z \to Z$, we define $h_{\nu}(\phi) = h_{\nu}(\phi^1)$.
And then $h_{\nu}(\phi^r) = \abs{r} h_{\nu}(\phi)$ for all $r \in \R$.

Now let $Z$ be a metric space and $T \colon Z \to Z$ be continuous.
Write $\mathcal{M}(T)$ for the set of $T$-invariant Borel probability measures on $Z$.
If $Z$ is compact then
\[\htop(T) = \sup_{\nu \in \mathcal{M}(T)} h_{\nu}(T).\]
(This result is called the variational principle.)
A \defn{measure of maximal entropy} is a measure $\nu \in \mathcal{M}(T)$ such that $h_{\nu}(T) = \htop(T)$, i.e. $h_{\nu}(T)$ realizes the supremum.

\begin{convention}
From now on, we will write $\phi = g_\G^1$, the time-one map of the geodesic flow on $\modG{SX}$.
On occasion, we will write $\tilde{\phi}$ for $g^1$ on $SX$.
\end{convention}

Recall the metric $d_{SX}$ on $SX$ is $d_{SX}(v,w) = \sup_{t \in \R} e^{-\abs{t}} d_{X}(v(t),w(t))$.
The quotient metric $d_{\modG{X}}$ on $\modG{X}$ is
\[d_{\modG{X}}(\piG{x},\piG{y}) = \inf_{\g \in \G} d_{X}(x,\g y),\]
and the quotient metric $d_{\modG{SX}}$ on $\modG{SX}$ is
\[d_{\modG{SX}}(\piGS{v},\piGS{w}) = \inf_{\g \in \G} d_{SX}(v,\g w).\]
We shall write $d$ for all these metrics---$d_{X}$, $d_{SX}$, $d_{\modG{X}}$, and $d_{\modG{SX}}$---as long as the context is clear.

For $k \ge 0$, we also define the metric $d_k$ on $SX$ by
\[d_k(v,w) = \sup_{0 \le t \le k} d(g^t v, g^t w),\]
and write $d_k$ for the corresponding quotient metric on $\modG{SX}$.
(Since $\modG{SX}$ is compact, $\htop(\phi)$ is the same whether computed using the metric $d$ or $d_k$.)

\begin{lemma}	[cf.~ Lemma 2.5 of \cite{knieper98}]	\label{knieper98 2.5}
Let $0 < \epsilon < \min \set{R_0, \injrad(X,\G)}$, where $R_0$ is the constant from \lemref{knieper98 2.3 5}.
Let $\mathcal{A} = \set{A_1, \dotsc, A_m}$ be a measurable partition of $\modG{SX}$ such that $\diam_{d_1} \mathcal{A} = \max_{A_i \in \mathcal{A}} \diam_{d_1} A_i < \epsilon$.
Then there is some constant $a > 0$ such that
\[m_\G(\alpha) \le e^{-\delta_\G n} a\]
for all $\alpha \in \mathcal{A}_{\phi}^{(n)}$.
\end{lemma}

\begin{proof}

Let $\alpha \in \mathcal{A}_{\phi}^{(n)}$ be arbitrary.
Since $\epsilon < \injrad(X,\G)$, we may lift $\alpha$ to a measurable set $\tilde{\alpha} \subset SX$ with $\diam_{d_n}(\tilde{\alpha}) < \epsilon$. 
Choose $v \in \tilde{\alpha}$, and let $p = v(0)$ and $x = v(n)$.
Let $w \in \tilde{\alpha}$ be arbitrary,
and put $\xi = w^-$.
Find $u \in SX$ such that $u(n) = x$ and $u^- = \xi$.
Let $q = u(0)$.
Since $d(v(0),w(0)) < \epsilon$ and $d(v(n),w(n)) < \epsilon$, we obtain
\[d(q,p)
  < d(u(0),w(0)) + \epsilon
\le d(u(n),w(n)) + \epsilon
  = d(v(n),w(n)) + \epsilon
  < 2\epsilon\]
by convexity.
Thus
$d(q,x) \ge d(x,p) - d(p,q) > n - 2\epsilon$,
and therefore
\[\mu_p(\pr_{\xi}(\Ball(x,\epsilon)))
\le e^{\delta_\G d(p,q)} \mu_q(\pr_{\xi}(\Ball(x,\epsilon)))
\le b e^{-\delta_\G d(q,x)}
\le e^{2 \delta_\G \epsilon} b e^{-\delta_\G n},\]
where $b > 0$ is the constant from \lemref{knieper98 2.3 6}.

Since $w$ was arbitrary, it follows from $d(q,p) < 2\epsilon$ that
$w^- = \xi \in \pr_{x}(\Ball(p,2\epsilon))$ for all $w \in \tilde{\alpha}$.
Thus
\[\emap(\tilde{\alpha}) \subseteq \bigcup_{\xi \in \pr_{x}(\Ball(p,2\epsilon))} \set{\xi} \times \pr_{\xi}(\Ball(x,\epsilon)).\]
Since no $w \in \tilde{\alpha}$ can spend more than $\epsilon$ time in $\tilde{\alpha}$, we find therefore that
\[m(\tilde{\alpha})
\le \epsilon \cdot \int_{\pr_{x}(\Ball(p,2\epsilon))} \left( \int_{\pr_{\xi}(\Ball(x,\epsilon))} e^{-\delta_\G (b_{\xi} + b_{\eta})(w(0),p)} \, d\mu_p(\eta) \right) d\mu_p(\xi),\]
where each $w \in \emap^{-1}(\xi,\eta)$.
But for all $w \in \tilde{\alpha}$, we know $w(0) \in \Ball(p,\epsilon)$, hence $(b_{\xi} + b_{\eta})(w(0),p) < 2\epsilon$ for all $w \in \tilde{\alpha}$.
Thus
\[m(\tilde{\alpha})
\le \epsilon e^{2 \delta_\G \epsilon} e^{2 \delta_\G \epsilon} b e^{-\delta_\G n} \int_{\pr_{x}(\Ball(p,2\epsilon))} \, d\mu_p(\xi)
\le \epsilon e^{2 \delta_\G (2\epsilon + R_0)} b e^{-\delta_\G n} \norm{\mu_p},
\]
and therefore $m_\G(\alpha) = m(\tilde{\alpha}) \le \epsilon e^{2 \delta_\G (2\epsilon + R_0)} b \norm{\mu_p}$.
Fixing an arbitrary $p_0 \in X$, we have $\norm{\mu_p} \le e^{2\delta_\G \diam(\modG{X})} \norm{\mu_{p_0}}$, which proves the lemma.
\end{proof}

\begin{theorem}	[cf. Theorem 2.6 of \cite{knieper98}]	\label{knieper98 2.6}
The Bowen-Margulis measure $m_\G$ is a measure of maximal entropy for the geodesic flow on $\modG{SX}$, i.e.,
\[h_{m_\G}(\phi) = \htop(\phi) = \delta_\G.\]
\end{theorem}

\begin{proof}
Let $\mathcal{A}$ be a measurable partition of $\modG{SX}$ with $\diam_{d_1} \mathcal{A} < \epsilon$.
By \lemref{knieper98 2.5},
\[H(\mathcal{A}_{\phi}^{(n)})
= \sum_{\alpha \in \mathcal{A}_{\phi}^{(n)}} - \log(m_\G(\alpha)) m_\G(\alpha)
\ge (\delta_\G n - \log a) \sum_{\alpha \in \mathcal{A}_{\phi}^{(n)}} m_\G(\alpha)
= \delta_\G n - \log a.\]
Thus $h(\phi, \mathcal{A}) \ge \delta_\G$.
By \thmref{main manning}, we now have
$\delta_\G = \htop(\phi) \ge h_{m_\G}(\phi) \ge h_{m_\G}(\phi, \mathcal{A})\ge \delta_\G$, proving \thmref{knieper98 2.6}.
\end{proof}

\section{Entropy expansive}

Let $T \colon V \to V$ be a self-homeomorphism of a metric space $V$.
For $x \in X$ and $\epsilon > 0$, define the \defn{Bowen ball}
\begin{equation*}
Z_{T,\epsilon}(x)
= \setp{y \in V}{d_k(T^{n} x, T^{n} y) \le \epsilon \text{ for all } n \in \Z}.
\end{equation*}
We say $T$ is \defn{$h$-expansive} (or, \defn{entropy expansive}) if there is some $\epsilon > 0$ (called an \defn{$h$-expansivity constant} for $T$) such that
\begin{equation*}
\sup_{x \in V} \htop(f,Z_{T,\epsilon}(x)) = 0.
\end{equation*}

The following result is due to Rufus Bowen \cite[Theorem 3.5]{bowen72}.

\begin{theorem} [Bowen]				\label{bowen}
Let $V$ be a finite-dimensional compact metric space and $\mathcal{A} = \set{A_1, \dotsc, A_m}$ a Borel partition with $\diam \mathcal{A} < \epsilon$.
Let $\nu$ be a $T$-invariant Borel probability measure on $V$.
If $T$ is $h$-expansive with expansivity constant $\epsilon$, then
\[h_{\nu}(T) = h_{\nu}(T,\mathcal{A}).\]
\end{theorem}

\begin{lemma} [cf.~ Proposition 3.3 of \cite{knieper98}]	\label{knieper98 3.3}
The time-$k$ map $\phi^k = g^k_\G$ of the geodesic flow on $\modG{SX}$ is $h$-expansive with $h$-expansivity constant $\epsilon = \frac{1}{3} \injrad(X,\G)$.
\end{lemma}

\begin{proof}
Let $v \in SX$ and $\bar{v} = \piGS{v}$.
Since $2 \epsilon < \injrad(X,\G)$, the Bowen ball
\[Z_{\phi^k, \epsilon}(\bar{v})
= \setp{\bar{w} \in \modG{SX}}{d_k(\phi^{nk} \bar{v}, \phi^{nk} \bar{w}) \le \epsilon \text{ for all } n \in \Z}\]
lifts to the Bowen ball
\begin{align*}
Z_{\tilde{\phi}^k, \epsilon}(v)
& = \smallsetp{w \in SX}{d_k(\tilde{\phi}^{nk} v, \tilde{\phi}^{nk} w) \le \epsilon \text{ for all } n \in \Z} \\
&= \setp{w \in SX}{d(v(t), w(t)) \le \epsilon \text{ for all } t \in \R},
\end{align*}
which is a subset of the set $P(v)$ of geodesics parallel to $v$, and therefore has $\htop(\tilde{\phi},Z_{\tilde{\phi}^k, \epsilon}(v)) = 0$ by convexity of the $\CAT(0)$ metric.
Apply \lemref{entropy of covers}.
\end{proof}

It follows that the entropy map $\nu \mapsto h_{\nu}(\phi)$ is upper semicontinuous on $\mathcal{M}(\phi)$.
Moreover, we have the following from \thmref{bowen} and \corref{SXG is fin-dim}.

\begin{corollary} [cf.~ Corollary 3.4 of \cite{knieper98}]	\label{knieper98 3.4}
Let $\mathcal{A}$ be a Borel partition of $\modG{SX}$ such that $\diam_{d_k} \mathcal{A} \le \frac{1}{3} \injrad(X,\G)$, and let $\nu \in \mathcal{M}(\phi)$.
Then
\[h_{\nu}(\phi^k) = h_{\nu}(\phi^k,\mathcal{A}).\]
\end{corollary}

From now on, we will write $h = \delta_\G$.

\section{Uniqueness}

We turn to the proof that the Bowen-Margulis measure $m_\G$ on $\modG{SX}$ is the unique measure of maximal entropy.

STANDING HYPOTHESIS:
Fix a compact set $K \subset X$ such that $X = \G K := \bigcup_{\g \in \G} \g K$.
Also fix $p \in K$, and let $R \ge R_0$, the constant from \lemref{knieper98 2.3 5}.

For each $R' \ge 2R + \frac{1}{2}$ such that $K \subseteq \Ball(p,R')$, and each $x \in X$, define
\[D(x,R',R) = \setp{v \in SX}{v(0) \in \Ball(p,R') \text{ and } v(r) \in \Ball(x,R) \text{ for some } r > 0}.\]

\begin{lemma}					\label{existence of F}
For each $n$, there exist points $x_1, \dotsc x_{k(n)} \in X$ and a partition $\mc{F}^{n} = \{ F_{1}^{n}, \dotsc, F_{k(n)}^{n} \}$ of $\pifp^{-1}(\Ball(p,R'))$ such that
\[D(x_i,R',R) \subseteq F_{i}^{n} \subseteq D(x_i,R',2R)\]
and $d(p,x_i) = 2n + 2R + R'$ for all $i$.
\end{lemma}

\begin{proof}
Set $r(n) = 2n + 2R + R'$, and let $P = \set{x_1, \dotsc, x_{k(n)}}$ be a maximal \defn{weakly $2R$-separated} set---that is, $d(x_i,x_j) \ge 2R$ for all $i \neq j$ and $P$ is maximal for this property---on the sphere $S(p,r(n))$ in $X$ of radius $r(n)$ about $p$.
Since $P$ is weakly $2R$-separated, the balls $\Ball(x_i,R)$ are pairwise disjoint; since $P$ is maximal, the balls $\Ball(x_i,2R)$ cover $S(p,r(n))$.
Thus the sets $D(x_i,R',R)$ are pairwise disjoint, and the sets $D(x_i,R',2R)$ cover $\pifp^{-1}(\Ball(p,R'))$.
A partition $\mc{F}^{n}$ with the desired property is now straightforward to construct.
\end{proof}

For each $n$, fix a partition $\mc{F}^{n}$ guaranteed by \lemref{existence of F}, and let $L_{i}^{n} = \piGS(F_{i}^{n})$.
Then $\mc{L}^{n} = \{ L_{i}^{n} \}$ covers $\modG{SX}$ because $K \subseteq \Ball(p,R')$ and $\pifp$ is onto.

\begin{lemma} [cf.~ Lemma 5.1 of \cite{knieper98}]	\label{knieper98 5.1}
Let $R' \ge 2R + \frac{1}{2}$ such that $K \subseteq \Ball(p,R')$.
There exists a constant $c' > 0$ such that for all $x \in X$ satisfying $d(p,x) \ge R' + R$,
\[m(D(x,R',R)) \ge c' e^{-h d(p,x)}.\]
\end{lemma}

\begin{proof}

Let $\ell > 0$ be the constant from \lemref{knieper98 2.3 5}, and let $x \in X$ satisfy $d(p,x) \ge R' + R$.
Let $\eta \in \pr_{x}(\Ball(p,R))$, and choose $w \in \pifp^{-1}(x)$ such that $w^- = \eta$.
By choice of $\eta$, we have $q := w(t_0) \in \Ball(p,R)$ for some $t_0 < 0$.
We claim $(\eta,\zeta) \in \emap(D(x,R',R))$ for every $\zeta \in \pr_{\eta}(\Ball(x,R))$.
So let $\zeta \in \pr_{\eta}(\Ball(x,R))$.
By choice of $\zeta$, there exists $v \in \emap^{-1}(\eta,\zeta) \cap \pifp^{-1}(\Ball(x,R))$.
Since $w(t_0) \in \Ball(p,R)$, we obtain $v(t_0) \in \Ball(p,2R)$ by convexity of the metric.
Thus $g^{t_{\phantom{}_0}} v \in D(x,R',R)$, which proves our claim.
Hence
\begin{align*}
\mu_{p}(\pr_{\eta}(\Ball(x,R)))
&\ge e^{-hR} \mu_{q}(\pr_{\eta}(\Ball(x,R))) \\
&\ge e^{-hR} \cdot \frac{1}{b} e^{-h d(q,x)}
\ge \left[ \frac{e^{-2hR}}{b} \right] e^{-h d(p,x)}
\end{align*}
by \lemref{knieper98 2.3 6}.
Now for each $\zeta \in \pr_{\eta}(\Ball(x,R))$, if $v \in D(x,R',R) \cap \emap^{-1}(\eta,\zeta)$ is chosen such that $v(0)$ is the closest point on $v$ to $p$, then we have at least $v([-\frac{1}{2},\frac{1}{2}]) \subset D(x,R',R)$.
Therefore,
\begin{align*}
\int_{D(x,R',R)} \, dm
&\ge 1 \cdot \int_{\pr_{x}(\Ball(p,R))} \int_{\pr_{\eta}(\Ball(x,R))} e^{-h \beta_{p}(\eta,\zeta)} \, d\mu_{p}(\zeta) d\mu_{p}(\eta) \\
&\ge e^{-2hR} \int_{\pr_{x}(\Ball(p,R))} \mu_{p}(\pr_{\eta}(\Ball(x,R))) \, d\mu_{p}(\eta) \\
&\ge e^{-2hR} \int_{\pr_{x}(\Ball(p,R))} \left[ \frac{e^{-2hR}}{b} \right] e^{-h d(p,x)} \, d\mu_{p}(\eta) \\
&\ge \left[ \frac{\ell}{b} e^{-4hR} \right] e^{-h d(p,x)}.
\qedhere
\end{align*}
\end{proof}

\begin{lemma} [cf.~ Lemma 5.3 of \cite{knieper98}]	\label{knieper98 5.3}
There exist constants $r,c > 0$ such that for all $n \in \N$, no $\bar{v} \in \modG{SX}$ lies in more than $r$ sets $L_{i}^{n}$, and $m_\G(L_{i}^{n}) \ge c e^{-2hn}$ for all $i$.
\end{lemma}

\begin{proof}
The number $r$ can be taken to be the number of $\g \in \G$ such that $\g K \cap \Ball(p,R')$ is nonempty.
Hence by \lemref{knieper98 5.1},
\[r \cdot m_\G(L_{i}^{n})
\ge m(F_{i}^{n})
\ge m(D(x_i,R',R))
\ge c' e^{-h d(p,x_i)}
= c' e^{-h (2R + R')} e^{-2hn}\]
for all $n \in \N$ and $i \in \set{1, 2, \dotsc, k(n)}$.
Thus we may take $c = \frac{c'}{r} e^{-h (2R + R')}$.
\end{proof}

We say a set $A$ is \defn{$(d_n, \epsilon)$-separated} if it is $(n,\epsilon)$-separated (under $\phi$) with respect to the metric $d_n$.

\begin{lemma} [cf.~ Lemma 5.2 of \cite{knieper98}]	\label{knieper98 5.2}
Fix $n \in \N$, and let $q \in X$ satisfy $d(q,p) \ge n + R + R'$.
Then the cardinality of a $(d_n,\epsilon)$-separated set of $D(q,R',R)$ is bounded from above by a constant $a = a(\epsilon,R',R)$ depending only on $\epsilon$, $R'$, and $R$.
\end{lemma}

\begin{proof}
Fix $\epsilon > 0$.
Let $R''$ be the $r$ guaranteed by \corref{uniform compact-open constants} for $\frac{\epsilon}{2}$.
Choose maximal $\frac{\epsilon}{2}$-separated subsets $Q_1 \subset \Ball(p,R'+R'')$ and $Q_2 \subset \Ball(q,R+R'')$.
For each $x \in Q_1$ and $y \in Q_2$, choose $v_{x,y} \in SX$ such that $v_{x,y}(-R'') = x$ and $v_{x,y}(d(x,y) - R'') = y$.
We claim the set $A_{n} = \setp{v_{x,y}}{x \in Q_1 \text{ and } y \in Q_2}$ is an $(n,\frac{\epsilon}{2})$-separated set for $D(q,R',R)$.
So let $w \in D(q,R',R)$ be arbitrary.
Then $w(0) \in \Ball(p,R'')$ and $w(t_0) \in \Ball(q,R)$, hence $w(-R'')$ and  $w(t_0 + R'')$ lie within distance $\frac{\epsilon}{2}$ of some $x \in Q_1$ and $y \in Q_2$, respectively.
By convexity, $d_X(w(t), v_{x,y}(t)) \le \frac{\epsilon}{2}$ for all $t \in [-R'',n+R'']$.
Thus $d_{SX}(g^t w, g^t v_{x,y}) \le \frac{\epsilon}{2}$ for all $t \in [0,n]$, proving the claim.
Therefore, $\bigcup_{v_{x,y} \in A_{n}} \cl{\Ball}_{d_n}(v_{x,y},\frac{\epsilon}{2}) \supseteq D(q,R',R)$.
Now any $\frac{\epsilon}{2}$-separated set in $\Ball(p,R'+R'')$ is, up to an isometry $\g \in \G$, a $\frac{\epsilon}{2}$-separated set in $\Ball(K,R'+R'')$; similarly for $\Ball(q,R+R'')$.
Hence $\card{A_{n}} \le \card{Q_1} \cdot \card{Q_2} \le s_1(\phi,\Ball(K,R'+R''),\frac{\epsilon}{2}) \cdot s_1(\phi,\Ball(K,R+R''),\frac{\epsilon}{2})$, which depends only on $\epsilon$, $R$, and $R'$.
Since any $\epsilon$-separated set in $D(q,R',R)$ can have at most $\card{A_{n}}$ elements, we have proved the lemma.
\end{proof}

\begin{lemma} [cf.~ Lemma 5.4 of \cite{knieper98}]	\label{knieper98 5.4}
Let $P = \set{\bar{v}_1, \dotsc, \bar{v}_{\ell}}$ be a maximal $(d_{2n},\epsilon)$-separated set in $\modG{SX}$, and let $\mc{B}$ be a partition of $\modG{SX}$ such that for each $B \in \mc{B}$ there is some $\bar{v}_j \in P$ such that
\[\Ball_{d_{2n}}(\bar{v}_j,\frac{\epsilon}{2}) \subseteq B \subseteq \Ball_{d_{2n}}(\bar{v}_j,\epsilon).\]
Then for each $L_{i}^{n} \in \mc{L}^{n}$,
\[\card{\setp{B \in \mc{B}}{B \cap L_{i}^{n} \neq \varnothing}}
\le a(\epsilon,R'+\epsilon,2R+\epsilon),\]
where $a(\epsilon,R'+\epsilon,2R+\epsilon)$ is the constant from \lemref{knieper98 5.2}.
\end{lemma}

\begin{proof}
Lift $P$ to a set $\tilde{P} = \set{v_1, \dotsc, v_{\ell}} \subset SX$ such that each $\piGS(v_j) = \bar{v}_j$, and $d_{2n}(v_j,F_{i}^{n}) < \epsilon$ whenever $d_{2n}(\bar{v}_j,L_{i}^{n}) < \epsilon$.
Then $\tilde{P}$ is $(d_{2n},\epsilon)$-separated, and
\begin{align*}
\card{\setp{B \in \mc{B}}{B \cap L_{i}^{n} \neq \varnothing}}
&\le \card{\setp{\bar{v}_j \in P}{B_{d_{2n}}(\bar{v}_j,\epsilon) \cap L_{i}^{n} \neq \varnothing}} \\
&= \card{\setp{v_j \in \tilde{P}}{d_{2n}(v_j,F_{i}^{n}) < \epsilon}}
\end{align*}
by hypothesis on $\mc{B}$ and choice of $\tilde{P}$.
But every geodesic $v \in SX$ that satisfies $d_{2n}(v,F_{i}^{n}) < \epsilon$ lies in $D(x_i,R'+\epsilon,2R+\epsilon)$, so by \lemref{knieper98 5.2} 
\[\card{\setp{v_j \in \tilde{P}}{d_{2n}(v_j,F_{i}^{n}) < \epsilon}}
\le a(\epsilon,R'+\epsilon,2R+\epsilon).\qedhere\]
\end{proof}

Let $\mc{A}^n = \set{\phi^n(L_i^{n})}$ for each $n \in \N$.

\begin{lemma} [cf.~ Lemma 5.6 of \cite{knieper98}]	\label{knieper98 5.6}
Let $\nu$ be a Borel probability measure on $\modG{SX}$
and $\Omega \subseteq \modG{SX}$ be a Borel set such that $\piGSi(\Omega)$ contains all geodesics of nonzero width.
Then there exists a union $C_n$ of sets $A \in \mc{A}^n = \set{\phi^n(L_i^{n})}$ such that
\[\nu(\Omega \Delta C_n) := \nu(\Omega \setminus C_n) + \nu(C_n \setminus \Omega) \to 0.\]
\end{lemma}

\begin{proof}
By construction of $\mc{L}^{n}$, for every $n \in \N$ and $v,w \in A \in \mc{A}^n$, there exist lifts $\tilde{v},\tilde{w}$ of $v,w$ (respectively) such that $d(\tilde{v}(t),\tilde{w}(t)) \le R'$ for all $t \in [-n,n]$.

Let $\delta > 0$ and choose compact sets $K_1^{\delta} \subset \Omega$ and $K_2^{\delta} \subset (\modG{SX}) \setminus \Omega$ satisfying $\nu(\Omega \setminus K_1^{\delta}) < \delta$ and $\nu((\modG{SX}) \setminus K_2^{\delta}) < \delta$.
Let $\tilde{K}_1^{\delta} = \piGSi (K_1^{\delta})$ and $\tilde{K}_2^{\delta} = \piGSi (K_2^{\delta})$.
We claim there exists $n_\delta \in \N$ such that for all $v \in \tilde{K}_1^{\delta}$ and $w \in \tilde{K}_2^{\delta}$, there exists $t \in [-n_\delta, n_\delta]$ such that $d(v(t),w(t)) > b$.
Otherwise, by compactness there exist $v_0 \in \tilde{K}_1^{\delta}$ and $w_0 \in \tilde{K}_2^{\delta}$ such that $d(v_0(t),w_0(t)) \le b$ for all $t \in \R$.
But then $v_0$ and $w_0$ bound a flat strip, contradicting the fact that $w_0$ has zero width.
Thus, when $n \ge n_\delta$, we see by our first observation that no $A \in \mc{A}^{n}$ intersects both $K_1^{\delta}$ and $K_2^{\delta}$.

Let $n \ge n_\delta$, and let $C_n^{\delta}$ be the union of all sets $A \in \mc{A}^{n}$ such that $A \cap K_1^{\delta} \neq \varnothing$.
Notice $K_1^{\delta} \subseteq C_n^{\delta}$ by construction, and $C_n^{\delta} \cap K_2^{\delta} = \varnothing$ by the previous paragraph.
Thus we see that
\[\nu(\Omega \Delta C_n^{\delta})
:= \nu(\Omega \setminus C_n^{\delta}) + \nu(C_n^{\delta} \setminus \Omega)
\le \nu(\Omega \setminus K_1^{\delta}) + \nu((C_n^{\delta} \setminus \Omega) \setminus K_2^{\delta})
\le \delta + \delta.\]
Therefore, for every $\delta > 0$ there exists a sequence $(C_n^{\delta})$ of measurable sets and $n_\delta \in \N$ such that for all $n \ge n_\delta$, $\nu(\Omega \Delta C_n^{\delta}) \le \delta$.
So take a decreasing sequence $\delta_k \to 0$.
Since $\nu(\Omega \Delta C_n^{\delta}) \le \delta$ for all $n \ge n_\delta$ (not just $n_\delta$ itself), we may assume that $n_{\delta_k}$ is strictly increasing in $k$.
Let $C_n = C_n^{\delta_{k_n}}$, where $k_n = \max \setp{k}{n_{\delta_k} \le n}$.
This is the desired sequence $(C_n)$.
\end{proof}

\begin{theorem} [cf.~ Theorem 5.8 of \cite{knieper98}]	\label{uniqueness of maximal entropy}
Let $\G$ be a group acting freely geometrically on a proper, geodesically complete $\CAT(0)$ space $X$ with rank one axis.
The Bowen-Margulis measure $m_\G$ on $\modG{SX}$ is the unique measure of maximal entropy (up to rescaling) for the geodesic flow $g_\G^t$ on $\modG{SX}$. \end{theorem}

\begin{proof}
It remains to show that if $\nu \in \mathcal{M}(\phi)$, then $\nu \neq m_\G$ implies
$h_{\nu}(\phi) < \htop(\phi)$.
So let $\nu \neq m_\G$ be an invariant Borel probability measure on $\modG{SX}$.
Since $m_\G$ is ergodic, it suffices to consider the case that $\nu$ and $m_\G$ are mutually singular.

Let $\epsilon = \frac{1}{3} \injrad(X,\G)$ and let $P$ be a maximal $(d_{2n},\epsilon)$-separated set in $\modG{SX}$.
Let $\mc{B}^n$ be a partition of $\modG{SX}$ such that for each $B \in \mc{B}^n$ there exists $\bar{v}_j \in P$ such that
$\cl{\Ball}_{d_{2n}}(\bar{v}_j,\frac{\epsilon}{2}) \subset B
\subset \cl{\Ball}_{d_{2n}}(\bar{v}_j,\epsilon)$.
By \corref{knieper98 3.4},
\begin{equation} \label{5.8 hnu}
2n h_{\nu}(\phi^1)
= h_{\nu}(\phi^{2n})
= h_{\nu}(\phi^{2n},\mc{B}^n)
\le H_{\nu}(\mc{B}^n)
= - \sum_{B \in \mc{B}^n} \nu(B) \log \nu(B).
\end{equation}
Since $\nu$ and $m_\G$ are mutually singular, there is a set $\Omega \subset \modG{SX}$ such that $m_\G(\Omega) = 0$ and $\nu(\Omega) = 1$.
Since $m$ gives zero measure to the set of all geodesics of nonzero width, we may assume $\piGSi(\Omega)$ contains every geodesic of nonzero width.
By \lemref{knieper98 5.6}, there are sets $C_n$ which are unions of elements of $\mc{A}^n$ and which satisfy $(\nu + m_{\G})(C_n \Delta \Omega) \to 0$ as $n \to \infty$.
In particular, \begin{equation} \label{5.8 near start}
\nu(C_n) \to 1 \text{ and } m_\G(C_n) \to 0 \text{ as } n \to \infty.
\end{equation}

It is a standard fact that if $a_1, \dotsc, a_k \ge 0$ satisfy $\sum_{i=1}^k a_i \le 1$, then
\begin{equation} \label{5.7}
- \sum_{i=1}^{k} a_i \log a_i
\le \left( \sum_{i=1}^{k} a_i \right) \log k + \frac{1}{e}.
\end{equation}
Thus, splitting
$\sum_{B \in \mc{B}^n} \nu(B) \log \nu(B)$
into two sums---gathering those $B \in \mc{B}^n$ for which $\phi^n(B) \cap C_n$ is nonempty or empty together---and applying \eqrefstar{5.7} to each sum, we find
\begin{equation*}
2n h_{\nu}(\phi)
\le b_n \log \card{\mc{B}_C^n} + (1 - b_n) \log \card{\mc{B}^n}
+ \frac{2}{e}
\end{equation*}
by \eqrefstar{5.8 hnu}, where
\[\mc{B}_C^n = \setp{B \in \mc{B}^n}{\phi^n(B) \cap C_n \neq \varnothing}
\qquad \text{and} \qquad
b_n = \sum_{B \in \mc{B}_C^n} \nu(B).\]
Since $C_n$ is a union of sets $A_i^{n} = \phi^{n}(L_i^{n})$, every $B \in \mc{B}_C^n$ satisfies $B \cap L_i^{n} \neq \varnothing$ for at least one $L_i^{n} \in \mc{L}^{n}$ such that $L_i^{n} \subseteq \phi^{-n}(C_n)$, and therefore
\begin{equation*}
\card{\mc{B}_C^n}
\le a(\epsilon, R' + \epsilon, 2R + \epsilon) \card{\setp{L_i^{n} \in \mc{L}^{n}}{L_i^{n} \subseteq \phi^{-n}(C_n)}}
\end{equation*}
by \lemref{knieper98 5.4}.
Now by \lemref{knieper98 5.3},
\begin{equation*}
\card{\mc{B}_C^n}
\le r \cdot a(\epsilon, R' + \epsilon, 2R + \epsilon) \frac{m_\G(\phi^{-n} C_n)}{\min m_\G(L_i^{n})}
\le c'' \cdot m_\G(C_n) e^{2hn},
\end{equation*}
where $c'' = \frac{r}{c} \cdot a(\epsilon, R' + \epsilon, 2R + \epsilon)$ is constant.
Thus we have shown
\begin{equation} \label{5.8 middle}
2n h_{\nu}(\phi)
\le b_n \log \left[ c'' \cdot m_\G(C_n) e^{2hn} \right]
+ (1 - b_n) \log \card{\mc{B}^n} + \frac{2}{e}.
\end{equation}
Repeating this argument with $\modG{SX}$ in place of $C_n$, we find
\[\card{\mc{B}^n}
\le r \cdot a(\epsilon, R' + \epsilon, 2R + \epsilon) \frac{1}{\min m_\G(L_i^{n})}
\le c'' e^{2hn},\]
and therefore by \eqrefstar{5.8 middle},
\begin{align*}
2n h_{\nu}(\phi)
&\le b_n \left[ 2hn + \log c'' + \log m_\G(C_n) \right]
+ (1 - b_n) \left[ 2hn + \log c'' \right] + \frac{2}{e} \\
&= 2hn + \log c'' + b_n \log m_\G(C_n) + \frac{2}{e}.
\end{align*}
In other words,
$2n h_{\nu}(\phi) - 2hn
\le \log c'' + b_n \log m_\G(C_n) + \frac{2}{e}$.
But we know
\[b_n = \sum_{B \in \mc{B}_C^n} \nu(B)
\ge \nu(C_n) \to 1 \text{ and } m_\G(C_n) \to 0\]
as $n \to \infty$ from \eqrefstar{5.8 near start}, so $b_n \log m_\G(C_n) \to -\infty$ as $n \to \infty$, forcing $h_{\nu}(\phi) < h$.
\end{proof}

\bibliographystyle{amsplain}
\bibliography{refs}

\end{document}

%% file: artcfg.tex
\usepackage{amssymb}
\usepackage{hyperref}
\usepackage{tikz}

\usepackage{environ}

\usepackage{import}

\theoremstyle{plain}
\newtheorem{theorem}[equation]{Theorem}
\newtheorem*{theorem*}{Theorem}

\newtheorem{corollary}[equation]{Corollary}
\newtheorem{lemma}[equation]{Lemma}
\newtheorem{conj}[equation]{Conjecture}

\newtheorem*{fact*}{Fact}
\newtheorem*{question*}{Question}

\newtheorem{main}{Theorem}

\newtheorem{maintheorem}[main]{Theorem}

\theoremstyle{remark}

\newtheorem*{remark}{Remark}

\theoremstyle{definition}
\newtheorem{definition}[equation]{Definition}
\newtheorem*{convention}{Convention}

\newtheorem*{standing hypothesis}{Standing Hypothesis}

\newcommand{\thmref}[1]{Theorem~\ref{#1}}
\newcommand{\corref}[1]{Corollary~\ref{#1}}
\newcommand{\lemref}[1]{Lemma~\ref{#1}}

\newcommand{\defref}[1]{Definition~\ref{#1}}

\newcommand{\defn}[1]{\emph{#1}}

\graphicspath{{pictures/}}

\newcommand{\N}{\mathbb{N}}
\newcommand{\Z}{\mathbb{Z}}

\newcommand{\R}{\mathbb{R}}

\renewcommand{\setminus}{\smallsetminus}

\DeclareMathOperator{\supp}{supp}

\DeclareMathOperator{\injrad}{injrad}
\DeclareMathOperator{\diam}{diam}

\DeclareMathOperator{\Isom}{Isom}

\newcommand{\set}[1]{\left\{#1\right\}}
\newcommand{\setp}[2]{\left\{#1 : #2\right\}}
\newcommand{\smallsetp}[2]{\{#1 : #2\}}
\newcommand{\family}[1]{\left(#1\right)}
\newcommand{\abs}[1]{\left| #1 \right|}
\newcommand{\norm}[1]{\left\| #1 \right\|}

\newcommand{\res}[1]{\vert_{#1}}

\newcommand{\cl}[1]{\overline{#1}}
\newcommand{\bd}{\partial}
\newcommand{\double}[1]{#1 \times #1}
\newcommand{\dbX}{\double{\bd X}}

\newcommand{\Reg}{\mathcal R}

\newcommand{\Zerowidth}{\mathcal{Z}}

\newcommand{\lmod}{\backslash}

\newcommand{\modgp}[2]{#2 \lmod #1}
\newcommand{\modG}[1]{\modgp{#1}{\Gamma}}

\DeclareMathOperator{\pr}{pr}

\DeclareMathOperator{\emap}{E}

\newcommand{\G}{\Gamma}

\newcommand{\g}{\gamma}

\DeclareMathOperator{\dT}{d_T}
\renewcommand{\epsilon}{\varepsilon}
\DeclareMathOperator{\width}{width}

\DeclareMathOperator{\CAT}{CAT}

\newcommand{\card}[1]{\# {#1}}

\newcommand{\mc}{\mathcal}

%% file: ricks-entropy.bbl
\providecommand{\bysame}{\leavevmode\hbox to3em{\hrulefill}\thinspace}
\providecommand{\MR}{\relax\ifhmode\unskip\space\fi MR }
\providecommand{\MRhref}[2]{%
  \href{http://www.ams.org/mathscinet-getitem?mr=#1}{#2}
}
\providecommand{\href}[2]{#2}
\begin{thebibliography}{10}

\bibitem{ballmann}
Werner Ballmann, \emph{Lectures on spaces of nonpositive curvature}, DMV
  Seminar, vol.~25, Birkh\"auser Verlag, Basel, 1995, With an appendix by Misha
  Brin. \MR{1377265 (97a:53053)}

\bibitem{bowen72}
Rufus Bowen, \emph{Entropy-expansive maps}, Trans. Amer. Math. Soc.
  \textbf{164} (1972), 323--331, \url{http://dx.doi.org/10.2307/1995978},.
  \MR{0285689}

\bibitem{bridson}
Martin~R. Bridson and Andr{\'e} Haefliger, \emph{Metric spaces of non-positive
  curvature}, Grundlehren der Mathematischen Wissenschaften [Fundamental
  Principles of Mathematical Sciences], vol. 319, Springer-Verlag, Berlin,
  1999. \MR{1744486 (2000k:53038)}

\bibitem{engelking}
Ryszard Engelking, \emph{Theory of dimensions finite and infinite}, Sigma
  Series in Pure Mathematics, vol.~10, Heldermann Verlag, Lemgo, 1995.
  \MR{1363947}

\bibitem{kleiner}
Bruce Kleiner, \emph{The local structure of length spaces with curvature
  bounded above}, Math. Z. \textbf{231} (1999), no.~3, 409--456,
  \url{http://dx.doi.org/10.1007/PL00004738},. \MR{1704987}

\bibitem{knieper98}
Gerhard Knieper, \emph{The uniqueness of the measure of maximal entropy for
  geodesic flows on rank {$1$} manifolds}, Ann. of Math. (2) \textbf{148}
  (1998), no.~1, 291--314, \url{http://dx.doi.org/10.2307/120995},. \MR{1652924
  (2000b:37016)}

\bibitem{ln19}
Alexander Lytchak and Koichi Nagano, \emph{Geodesically complete spaces with an
  upper curvature bound}, Geom. Funct. Anal. \textbf{29} (2019), no.~1,
  295--342, \url{http://dx.doi.org/10.1007/s00039-019-00483-7},. \MR{3925112}

\bibitem{manning}
Anthony Manning, \emph{Topological entropy for geodesic flows}, Ann. of Math.
  (2) \textbf{110} (1979), no.~3, 567--573,
  \url{http://dx.doi.org/10.2307/1971239},. \MR{554385}

\bibitem{ricks-mixing}
Russell Ricks, \emph{Flat strips, {B}owen-{M}argulis measures, and mixing of
  the geodesic flow for rank one {${\rm CAT}(0)$} spaces}, Ergodic Theory
  Dynam. Systems \textbf{37} (2017), no.~3, 939--970,
  \url{http://dx.doi.org/10.1017/etds.2015.78},. \MR{3628926}

\bibitem{swenson99}
Eric~L. Swenson, \emph{A cut point theorem for {${\rm CAT}(0)$} groups}, J.
  Differential Geom. \textbf{53} (1999), no.~2, 327--358,
  \url{http://projecteuclid.org/euclid.jdg/1214425538},. \MR{1802725}

\bibitem{walters}
Peter Walters, \emph{An introduction to ergodic theory}, Graduate Texts in
  Mathematics, vol.~79, Springer-Verlag, New York, 1982. \MR{648108
  (84e:28017)}

\end{thebibliography}
